\documentclass[11pt,a4paper,reqno]{amsart}
\pdfoutput=1

\usepackage[T1]{fontenc}
\usepackage[utf8]{inputenc}
\usepackage{lmodern}
\usepackage[english]{babel}
\usepackage{csquotes}
\usepackage[activate]{pdfcprot}

\usepackage{datetime}

\usepackage{amssymb}
\usepackage{amscd}
\usepackage{amsfonts}
\usepackage{mathtools}

\usepackage{dsfont}
\usepackage{tikz-cd}

\usepackage[pdftex,colorlinks,linkcolor=black,citecolor=black,filecolor=black]{hyperref}

\renewcommand\labelenumi{(\roman{enumi})}
\renewcommand\theenumi\labelenumi

\numberwithin{equation}{section}
\theoremstyle{plain}
\newtheorem{theorem}{Theorem}[section]
\newtheorem{proposition}[theorem]{Proposition}
\newtheorem{lemma}[theorem]{Lemma}
\newtheorem{corollary}[theorem]{Corollary}
\theoremstyle{remark}
\newtheorem{remarkx}[theorem]{Remark}
\newenvironment{remark}
{\pushQED{\qed}\remarkx}
 {\popQED\endremarkx}
\theoremstyle{definition}
\newtheorem{definition}[theorem]{Definition}
\newtheorem{question}[theorem]{Question}
\newtheorem{examplex}[theorem]{Example}
\newenvironment{example}
{\pushQED{\qed}\examplex}
 {\popQED\endexamplex}

\newcommand\calS{\mathcal{S}}
\newcommand\calH{\mathcal{H}}
\newcommand\G{\mathcal{G}}
\newcommand\g{\mathfrak{g}}
\DeclareMathOperator{\Aut}{Aut}
\DeclareMathOperator{\supp}{supp}
\DeclareMathOperator{\Ad}{Ad}
\DeclareMathOperator{\SL}{SL}
\DeclareMathOperator{\SO}{SO}
\DeclareMathOperator{\hght}{ht}
\DeclareMathOperator{\dist}{d}

\newcommand\N{\mathbb{N}}
\newcommand\Q{\mathbb{Q}}
\newcommand\R{\mathbb{R}}
\newcommand\Z{\mathbb{Z}}

\newcommand{\dd}{\mathop{}\!\mathrm{d}}

\DeclarePairedDelimiter\br{(}{)}

\DeclarePairedDelimiter\abs{\lvert}{\rvert}
\DeclarePairedDelimiter\norm{\lVert}{\rVert}

\providecommand\for{}
\newcommand\SetSymbol[1][]{%
\nonscript\:#1\vert
\allowbreak
\nonscript\:
\mathopen{}}
\DeclarePairedDelimiterX\set[1]{\lbrace}{\rbrace}{%
\renewcommand\for{\SetSymbol[\delimsize]}
#1
}

\renewcommand{\epsilon}{\varepsilon}

\title[On convergence of random walks]{Aspects of convergence of random walks on finite volume homogeneous spaces}

\author{Roland Prohaska}
\address{Departement Mathematik, ETH Z\"{u}rich, R\"{a}mistrasse 101, 8092 Z\"{u}rich, Switzerland}
\email{roland.prohaska@math.ethz.ch}

\date{\usdate\today}

\begin{document}

\begin{abstract}
We investigate three aspects of weak* convergence of the $n$-step distributions of random walks on finite volume homogeneous spaces $G/\Gamma$ of semisimple real Lie groups. 
First, we look into the obvious obstruction to the upgrade from Ces\`aro to non-averaged convergence: periodicity. 
We give  examples where it occurs and conditions under which it does not. 
In a second part, we prove convergence towards Haar measure with exponential speed from almost every starting point. 
Finally, we establish a strong uniformity property for the Ces\`aro convergence towards Haar measure for uniquely ergodic random walks. 
\end{abstract}
\maketitle
\section{Introduction}\label{sec:intro}
Let $G$ be a real Lie group and $\Gamma$ a lattice in $G$, that is, a discrete subgroup of $G$ such that the homogeneous space $X=G/\Gamma$ admits a $G$-invariant Borel probability measure $m_X$. 
This measure $m_X$ is unique and we refer to it as the (normalized) \emph{Haar measure} on $X$. 
A good example to have in mind is $G=\SL_d(\R)$ and $\Gamma=\SL_d(\Z)$. 

The objects of study in this paper are \emph{random walks} on $X$, given by probability measures $\mu$ on $G$: 
A step corresponds to randomly choosing a group element $g\in G$ according to $\mu$ and then moving from the current location $X\ni x$ to $gx$. 
Starting at $x_0\in X$, the distribution of the location after $n$ steps is given by the convolution 
\begin{align}\label{n-step-dist}
\mu^{*n}*\delta_{x_0},
\end{align}
which is the push-forward of the product measure $\mu^{\otimes n}\otimes\delta_{x_0}$ under the multiplication map $G^n\times X\ni (g_n,\dots,g_1,x)\mapsto g_n\dotsm g_1x\in X$. 

The broader context in which the study of these random walks originated is that of subgroup actions on homogeneous spaces. After Ratner's treatment of the rigidity and asymptotic properties of unipotent actions in her celebrated series of articles~\cite{rat2,rat1,rat3,rat4}, a new approach was needed to understand the dynamics of non-unipotent actions. Passing from a deterministic to a probabilistic point of view turned out to be a particularly fruitful angle. Still, understanding the long-term behavior of random walks on homogeneous spaces and the limiting behavior of the $n$-step distributions \eqref{n-step-dist} is a notoriously difficult problem. Major contributions to this line of study were made e.g.\ by Eskin--Margulis in their work on non-divergence~\cite{EM}, and by Benoist--Quint in their breakthrough series of articles~\cite{BQ1,BQ_rec,BQ2,BQ3}. 
We reproduce one of the main results of~\cite{BQ3} as motivating example. 
For the statement, recall that a probability measure $\nu$ on $X$ is called \emph{homogeneous} if there exists a closed subgroup $H$ of $G$ and a point $x\in X$ such that $\supp(\nu)=Hx$ is a closed orbit and $\nu$ is $H$-invariant. 
\begin{theorem}[{Benoist--Quint~\cite{BQ3}}]\label{thm:BQ}
Let $\mu$ be a compactly supported probability measure on $G$. 
Denote by $\calS$ and $\G$ the closed subsemigroup and subgroup of $G$ generated by $\supp(\mu)$, respectively, and suppose that the Zariski closure of $\Ad(\G)$ in $\Aut(\g)$ is Zariski connected, semisimple, and has no compact factors. 
Then for every $x_0\in X$ there is a homogeneous probability measure $\nu_{x_0}$ on $X$ with $\supp(\nu_{x_0})=\overline{\calS x_0}=\overline{\G x_0}$ and such that 
\begin{align}\label{cesaro}
\frac1n\sum_{k=0}^{n-1}\mu^{*k}*\delta_{x_0}\longrightarrow \nu_{x_0}
\end{align}
as $n\to\infty$ in the weak* topology. 
\end{theorem}
Here the weak* convergence \eqref{cesaro} more explicitly means that for every compactly supported continuous function $f\in C_c(X)$ we have 
\begin{align*}
\frac1n\sum_{k=0}^{n-1}\int_X f\dd(\mu^{*k}*\delta_{x_0})=\frac1n\sum_{k=0}^{n-1}\int_{G^k} f(g_k\dotsm g_1x_0)\dd\mu^{\otimes k}(g_1,\dots,g_k)\longrightarrow \int_Xf\dd\nu_{x_0}
\end{align*}
as $n\to\infty$. 
Recently, it was shown by B\'enard--de Saxc\'e~\cite{benard-desaxce} that the compact support assumption on $\mu$ in Theorem~\ref{thm:BQ} can be relaxed to a finite first moment assumption; see  Remark~\ref{rmk:first_moment}. Another recent generalization of the theorem above in joint work of the author with C.\ Sert and R.\ Shi~\cite{expanding} replaces the algebraic assumption on the support of $\mu$ by a certain expansion condition, which allows for cases in which $\mu$ is e.g.\ supported on a parabolic subgroup of a semisimple group.

Some questions left open by Theorem~\ref{thm:BQ} are listed by Benoist--Quint at the end of their survey~\cite{BQ:intro}. 
A major one is the following. 
\begin{question}\label{q}
In the setting of Theorem~\ref{thm:BQ}, is it also true that 
\begin{align}\label{non-averaged}
\mu^{*n}*\delta_{x_0}\longrightarrow \nu_{x_0}
\end{align}
as $n\to\infty$? 
\end{question}
Answers are available only in special cases: 
Breuillard~\cite{Br} established \eqref{non-averaged} for certain measures $\mu$ supported on unipotent subgroups, Buenger~\cite{Bue} proved it for some sparse solvable measures, and in previous work the author dealt with the case of spread out measures~\cite{spread_out}. Very recently, B\'enard~\cite{benard.foguel} observed that \eqref{non-averaged} holds for aperiodic measures $\mu$ under the assumption that $\mu$ has two convolution powers which are not mutually singular.

The purpose of this article is to discuss three (largely independent) aspects of random walk convergence related to Theorem~\ref{thm:BQ} and Question~\ref{q}, mainly having in mind the case that $G$ is a semisimple real Lie group. 
We are going to use the following terminology. 
\begin{definition}
Let $\nu$ be a probability measure on $X$ and $x_0\in X$. 
We say that the random walk on $X$ given by $\mu$ \emph{converges to $\nu$ on average} (resp.\ \emph{converges to $\nu$}) from the starting point $x_0$ if $\frac1n\sum_{k=0}^{n-1}\mu^{*k}*\delta_{x_0}\rightarrow \nu$ (resp.\ $\mu^{*n}*\delta_{x_0}\rightarrow \nu$) as $n\to\infty$ in the weak* topology. 
\end{definition}
Convergence on average is also commonly referred to as \emph{Ces\`aro convergence}. 
We use the two terms interchangeably. 

The article is organized as follows. 

In~\S\ref{sec:periodicity}, we look into the obvious obstruction to the upgrade from Ces\`aro convergence to (non-averaged) convergence: periodicity. 
We show in Example~\ref{ex:example1} how \eqref{non-averaged} can fail when $x_0$ has finite orbit under $\calS$. 
Using a product construction, we can also produce a counterexample in which the orbit closure $\overline{\calS x_0}$ has positive dimension (Example~\ref{ex:example2}). 
In both cases, the periodic behavior occurs at the level of the connected components of the orbit closure. 
As it turns out, this is no coincidence: 
If, in the setting of Theorem~\ref{thm:BQ}, the orbit closure $\overline{\calS x_0}$ is connected, there can be no periodicity (Theorem~\ref{thm:aperiodicity}) and we can show that the Ces\`aro convergence \eqref{cesaro} also holds along arithmetic progressions (Corollary~\ref{cor:arithmetic_cesaro}). 

In~\S\ref{sec:spectral_gap}, we establish effective convergence of random walks to the normalized Haar measure $m_X$ for typical starting points $x_0$: 
When $\supp(\mu)$ generates a Zariski dense subgroup of a semisimple real Lie group $G$ without compact factors, for any fixed $L^2$-function $f$ on $X$ the convergence 
\begin{align*}
\int_X f\dd(\mu^{*n}*\delta_{x_0})\overset{n\to\infty}{\longrightarrow}\int_Xf\dd m_X
\end{align*}
not only holds but is in fact exponentially fast for $m_X$-almost every $x_0\in X$ (Theorem~\ref{thm:spectral_gap}, Proposition~\ref{prop:spectral_gap_example}). 
The proof relies on an $L^2$-spectral gap of the convolution operator 
\begin{align*}
\pi(\mu)\colon f\mapsto \Bigl(x\mapsto \int_Gf(gx)\dd\mu(g)\Bigr)
\end{align*}
acting on measurable functions on $X$. 
Taking into account regularity of the function $f$, the above can be further strengthened to the statement that almost every $x\in X$ is \emph{exponentially generic} (Definition~\ref{def:exp_generic}): 
Up to a constant factor depending on derivatives of $f$, the exponential speed of convergence holds uniformly over all compactly supported smooth functions (Theorem~\ref{thm:exp_generic}). 
Key to this upgrade are the definition of suitable Sobolev norms and a functional analytic argument involving relative traces, first exploited in a dynamical context by Einsiedler--Margulis--Venkatesh~\cite{EMV}. 

Finally, in~\S\ref{sec:uniform} we prove that convergence on average to $m_X$ happens locally uniformly in $x_0$ in a strong way when the random walk is uniquely ergodic and admits a Lyapunov function (Theorem~\ref{thm:lyapunov_uniform}). 
For example, this is the case when $G$ is a connected semisimple real algebraic group and $\supp(\mu)$ generates a non-discrete Zariski dense subgroup, and also in the setup of Simmons--Weiss~\cite{SW}, which has connections to Diophantine approximation problems on fractals. 
To this end, we introduce the new concept of \emph{$(K_n)_n$-uniform recurrence} (Definition~\ref{def:Kn_uniform}), which refines recurrence properties of random walks previously studied in~\cite{BQ_rec,EM}. 
\subsection*{Standing Assumptions \& Notation}
As many of our arguments work in greater generality, in the remainder of the article we will relax the assumptions stated at the beginning of this introduction. 
The following setup shall be in place whenever nothing else is specified: 
$G$ is a locally compact $\sigma$-compact metrizable group acting ergodically on a locally compact $\sigma$-compact metrizable space $X$ endowed with a $G$-invariant probability measure $m_X$; and $\mu$ is a Borel probability measure on $G$. 
\subsection*{Acknowledgments}
The author would like to express his gratitude to Andreas Wieser for valuable comments on preliminary versions of the article, and to Manfred Einsiedler for explaining how relative traces can be used to make separability effective. 
Thanks also go to HIM Bonn and the organizers of the trimester program \enquote{Dynamics: Topology and Numbers}, in the course of which parts of this manuscript were completed, for hospitality and providing an excellent working environment. Finally, the author is grateful to the anonymous referee for pointing out a simple way to establish a better speed of convergence in Theorems~\ref{thm:spectral_gap} and~\ref{thm:exp_generic}.
\section{Periodicity}\label{sec:periodicity}
In this section, we start with two simple counterexamples to \eqref{non-averaged}, which illustrate ways in which a random walk may exhibit periodic behavior (\S\ref{subsec:per_examples}). 
Analyzing these examples for their common feature, we are led to a simple condition ensuring aperiodicity, stated and proved in \S\ref{subsec:aperiodicity}. 

\subsection{Examples}\label{subsec:per_examples}
The first example with periodicity is on finite periodic orbits. 
In the following, for $d\ge 2$ we denote by $\mathbf{1}_d$ the $d\times d$-identity matrix. 
\begin{example}\label{ex:example1}
Consider the \emph{principal congruence lattice} 
\begin{align*}
\Gamma=\Gamma(2)=\set{g\in\SL_2(\Z)\for g\equiv\mathbf{1}_2\bmod 2}
\end{align*}
in $G=\SL_2(\R)$. 
Being the kernel of the reduction homomorphism from $\SL_2(\Z)$ to $\SL_2(\Z/2\Z)$, we recognize $\Gamma(2)$ as a finite-index normal subgroup of $\SL_2(\Z)$. 
In particular, $\Gamma(2)$ is a lattice in $G$. 
Let $\mu=\tfrac12(\delta_{h_1}+\delta_{h_2})$ with 
\begin{align*}
h_1=\begin{pmatrix}1&1\\0&1\end{pmatrix},\,h_2=\begin{pmatrix}1&0\\1&1\end{pmatrix}.
\end{align*}
Then the closed subgroup $\G$ generated by $\supp(\mu)=\set{h_1,h_2}$ is $\G=\SL_2(\Z)$, which is Zariski dense in $G$. 
The $\G$-orbit of $x_0=\mathbf{1}_2\Gamma\in G/\Gamma$ is 
\begin{align*}
\mathcal{O}=\Bigl\{&x_0,h_1x_0,h_2x_0,h_2h_1x_0=\bigl(\begin{smallmatrix}1&1\\1&2\end{smallmatrix}\bigr)x_0,h_1h_2x_0=\bigl(\begin{smallmatrix}2&1\\1&1\end{smallmatrix}\bigr)x_0,\\
&h_1h_2h_1x_0=h_2h_1h_2x_0=\bigl(\begin{smallmatrix}2&-1\\1&0\end{smallmatrix}\bigr)x_0\Bigr\},
\end{align*}
with transitions as shown in the following diagram: 
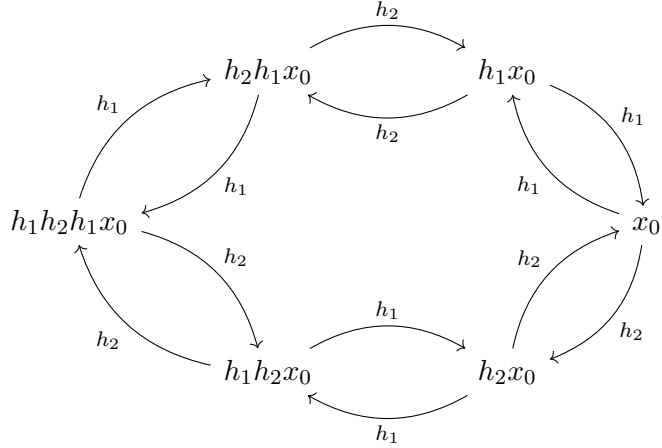
\begin{center}
\begin{tikzcd}
&h_2h_1x_0\arrow[rr,bend left,"h_2"]\arrow[ddl,bend left,"h_1"]&&h_1x_0\arrow[ddr,bend left,"h_1"]\arrow[ll,bend left, "h_2"]&\\
\\
h_1h_2h_1x_0\arrow[uur,bend left,"h_1"]\arrow[ddr,bend left,"h_2"]&&&&x_0\arrow[uul,bend left,"h_1"]\arrow[ddl,bend left,"h_2"]\\
\\
&h_1h_2x_0\arrow[rr,bend left,"h_1"]\arrow[uul,bend left,"h_2"]&&h_2x_0\arrow[uur,bend left,"h_2"]\arrow[ll,bend left, "h_1"]&
\end{tikzcd}
\end{center}
Consequently, we see that the random walk with starting point $x_0$ alternates between the two sets 
\begin{align*}
\mathcal{O}_1=\set{x_0,h_1h_2x_0,h_2h_1x_0}\text{ and }\mathcal{O}_2=\set{h_1x_0,h_2x_0,h_1h_2h_1x_0}.
\end{align*}
The $2$-step random walks on these sets constitute irreducible, aperiodic, finite state Markov chains, so that 
\begin{align*}
\mu^{*2n}*\delta_{x_0}&\longrightarrow\frac{1}{3}\sum_{p\in \mathcal{O}_1}\delta_p,\\
\mu^{*(2n+1)}*\delta_{x_0}&\longrightarrow\frac{1}{3}\sum_{p\in \mathcal{O}_2}\delta_p,
\end{align*}
as $n\to\infty$ in the weak* topology. 
\end{example}
In the example above, the support of $\mu$ generates a Zariski dense subgroup of $G$ and the lattice $\Gamma$ in $G$ is irreducible. 
(Recall that, loosely speaking, \enquote{irreducibility} of $\Gamma$ means that it does not arise from a product construction, cf.\ \cite[Definition~5.20]{discrete_subgroups}.) 
By the work of Benoist--Quint (\cite[Corollary~1.8]{BQ3}), these properties force any orbit closure $\overline{\calS x_0}$ to be either finite or all of $X$. 
As soon as intermediate orbit closures are possible, however, one can also construct examples with periodic behavior on non-discrete orbit closures. 
\begin{example}\label{ex:example2}
Let $G$, $\Gamma$, $X=G/\Gamma$, $h_1,h_2$, $x_0$ and $\G$ be as in Example~\ref{ex:example1} and choose a diagonal matrix $a\in\SL_2(\R)$ such that the diagonal entries of $a^2$ are irrational. 
We are going to consider the random walk on the product space 
\begin{align*}
X\times X=(G\times G)/(\Gamma\times\Gamma)
\end{align*}
given by the probability measure $\mu=\tfrac14\sum_{i=1}^4\delta_{g_i}$ on $G\times G$ with 
\begin{align*}
g_1&=(h_1,ah_1a^{-1}),\,g_2=(h_1,\mathbf{1}_2),\\
g_3&=(h_2,ah_2a^{-1}),\,g_4=(h_2,\mathbf{1}_2).
\end{align*}
The (closed) subgroup generated by the support of this measure $\mu$ is given by $\G\times a\G a^{-1}=\SL_2(\Z)\times a\SL_2(\Z)a^{-1}$. 
Indeed, the correct entry in the second copy of $G$ can be arranged using a finite product of $g_1^{\pm 1},g_3^{\pm 1}$, and then the entry in the first copy can be corrected using $g_2^{\pm 1},g_4^{\pm 1}$. 
By Theorem~\ref{thm:BQ} we thus know that for the starting point $(x_0,x_0)\in X\times X$ we have the weak* convergence 
\begin{align*}
\frac1n\sum_{k=0}^{n-1}\mu^{*k}*\delta_{(x_0,x_0)}\longrightarrow \nu_{(x_0,x_0)}
\end{align*}
as $n\to\infty$, where $\nu_{(x_0,x_0)}$ is the homogeneous probability measure on the closure of the $\G\times a\G a^{-1}$-orbit of $(x_0,x_0)$. 
(Recall that it makes no difference for the closure whether one considers the orbit under the generated subgroup or subsemigroup.) 

Let us identify this orbit closure. 
In the first copy of $X$, we recognize the finite orbit $\mathcal{O}$ from Example~\ref{ex:example1}. 
In the second copy, we see the action of irrational conjugates of $h_1,h_2$. 
As the acting group has product structure, the orbit closure in question is the product of these two orbit closures in the components: 
\begin{align*}
\overline{(\G\times a\G a^{-1})(x_0,x_0)}=\mathcal{O}\times \overline{a\G a^{-1} x_0}.
\end{align*}
Since the orbit $a\G a^{-1} x_0$ is infinite by our choice of the matrix $a$, it follows from \cite[Corollary~1.8]{BQ3} that $ \overline{a\G a^{-1} x_0}=X$, so that 
\begin{align*}
\overline{(\G\times a\G a^{-1})(x_0,x_0)}=\mathcal{O}\times X\text{ and }\nu_{(x_0,x_0)}= m_{\mathcal{O}}\otimes m_X
\end{align*}
for the normalized counting measure $m_{\mathcal{O}}$ on $\mathcal{O}$ and the normalized Haar measure $m_X$ on $X$. 
However, in analogy to Example~\ref{ex:example1}, the random walk is found to alternate between the sets 
\begin{align*}
\mathcal{O}_1\times X\text{ and }\mathcal{O}_2\times X,
\end{align*}
in the sense that $\supp(\mu^{*2n}*\delta_{(x_0,x_0)})\subset \mathcal{O}_1\times X$ and $\supp(\mu^{*(2n+1)}*\delta_{(x_0,x_0)})\subset \mathcal{O}_2\times X$ for all $n\in\N$. 
Hence, we conclude that the random walk starting from $(x_0,x_0)$ does not converge to $\nu_{(x_0,x_0)}$. 
\end{example}
\begin{remark}
The same behavior as in the previous example can be arranged inside a homogeneous space $X'=G'/\Gamma'$ that is the quotient of a semisimple real Lie group $G'$ by an irreducible lattice $\Gamma'$. 
Indeed, this is only a matter of choosing suitable embeddings $G\times G\hookrightarrow G'$ and $X\times X\hookrightarrow X'$, where $G$ and $X$ are as in Example~\ref{ex:example2}. 
Concretely, one can e.g.\ consider the $4\times 4$-congruence lattice 
\begin{align*}
\Gamma'=\Gamma(2)=\set{g\in\SL_4(\Z)\for g\equiv\mathbf{1}_4\bmod 2} 
\end{align*}
in $G'=\SL_4(\R)$ and the diagonal embeddings 
\begin{align*}
\begin{aligned}
G\times G&\hookrightarrow G',\\
(g,h)&\mapsto\begin{pmatrix}g&\\&h\end{pmatrix},
\end{aligned}
\qquad
\begin{aligned}
X\times X&\hookrightarrow X',\\
(g\Gamma,h\Gamma)&\mapsto\begin{pmatrix}g&\\&h\end{pmatrix}\Gamma'.
\end{aligned}
\end{align*}
We therefore see that Example~\ref{ex:example2}, i.e.\ periodic behavior on a non-discrete orbit closure, can be realized inside $X'=G'/\Gamma'$. 
Of course, after applying this embedding, the subgroup generated by the support of $\mu$ will no longer be Zariski dense in $G'$. 
\end{remark}

\subsection{An Aperiodicity Criterion}\label{subsec:aperiodicity}
Inspecting the examples above, one may notice that their common salient feature is that the orbit closure $\overline{\calS x_0}$ is disconnected. 
This naturally raises the question whether periodic behavior can also occur when this orbit closure is connected. 
In what follows, we answer this question in the negative. 
We shall use the following formalization of periodicity. 
\begin{definition}\label{def:periodicity}
Assume that the random walk on $X$ given by $\mu$ converges on average to a probability measure $\nu$ on $X$ from the starting point $x_0\in X$. 
We say that this convergence is \emph{periodic} if there exists an integer $d\ge 2$ and pairwise disjoint measurable subsets $D_0,\dots,D_{d-1}\subset X$ with $\nu(\partial D_i)=0$ for $0\le i< d$ and such that $(\mu^{*n}*\delta_{x_0})(D_{n\bmod d})=1$ for every $n\in \N$. 
Otherwise, we call the convergence \emph{aperiodic}. 
\end{definition}
The requirement on the boundaries of the sets $D_i$ is needed to ensure that the cyclic behavior is witnessed by the limit measure $\nu$. 
Without a condition of this sort, one could try to artificially define $D_i$ as the set of all points in $X$ that can be reached from $x_0$ precisely in $n\equiv i\bmod d$ steps. 
Indeed, this construction is possible for example when $\mu$ is finitely supported with the property that its support freely generates a discrete subsemigroup $\calS$ of $G$ and the starting point $x_0\in X$ has a free $\calS$-orbit. 
The latter is the case e.g.\ for $X=\SL_2(\R)/\SL_2(\Z)$, $\mu=\tfrac12(\delta_{h_1}+\delta_{h_2})$ with $h_1=(\begin{smallmatrix}1&2\\0&1\end{smallmatrix})$ and $h_2=(\begin{smallmatrix}1&0\\2&1\end{smallmatrix})$, and $x_0=a\SL_2(\Z)$ for a diagonal matrix $a\in\SL_2(\R)$ such that the diagonal entries of $a^2$ are irrational. 

We are now ready to state the announced aperiodicity theorem. 
\begin{theorem}\label{thm:aperiodicity}
Retain the notation and assumptions from Theorem~\textup{\ref{thm:BQ}} and let $x_0\in X$ be such that the orbit closure $\overline{\calS x_0}$ is connected. 
Then the Ces\`aro convergence to $\nu_{x_0}$ of the random walk on $X$ given by $\mu$ starting from $x_0$ is aperiodic. 
\end{theorem}
For the proof we need the following simple lemma. 
\begin{lemma}\label{lem:power}
Let $H$ be a Zariski connected real algebraic group and $S$ a subset of $H$ generating a Zariski dense subsemigroup. 
Then for every $d\in\N$, also the $d$-fold product set $S^d=\set{g_d\dotsm g_1\for g_1,\dots,g_d\in S}$ generates a Zariski dense subsemigroup of $H$. 
In particular, if $\supp(\mu)$ generates a Zariski dense subsemigroup for some probability measure $\mu$ on $H$, the same is true for $\supp(\mu^{*d})$. 
\end{lemma}
\begin{proof}
Let $U\subset H$ be a non-empty Zariski open subset and consider the map $\phi\colon H\to H,\,g\mapsto g^d$. 
Since $\phi$ is Zariski continuous, $\phi^{-1}(U)$ is Zariski open. 
Moreover, this preimage is non-empty because $U$ is dense in the Lie group topology and $\phi$ is a diffeomorphism near the identity. 
By the assumption that $S$ generates a Zariski dense subsemigroup, we can thus find an element $g\in\phi^{-1}(U)$ that is the product of finitely many elements of $S$. 
It follows that $\phi(g)=g^d$ lies in the intersection of $U$ with the subsemigroup generated by $S^d$. 

The second claim involving $\mu$ immediately follows from the above together with the inclusion $\supp(\mu^{*d})\supset \supp(\mu)^d$. 
\end{proof}
\begin{proof}[Proof of Theorem~\textup{\ref{thm:aperiodicity}}]
Suppose $d\in\N$ is an integer such that there are pairwise disjoint $D_0,\dots,D_{d-1}\subset X$ with $\nu_{x_0}(\partial D_i)=0$ for all $0\le i< d$ and such that $(\mu^{*n}*\delta_{x_0})(D_{n\bmod d})=1$ for all $n\in\N$ as in the definition of periodicity. 
We have to show that $d=1$. 

First note that from Theorem~\ref{thm:BQ} and the properties of the sets $D_i$ it follows that 
\begin{align}\label{first_value}
\nu_{x_0}(D_0)=\lim_{n\to\infty}\frac1n\sum_{k=0}^{n-1}(\mu^{*k}*\delta_{x_0})(D_0)=\frac1d,
\end{align}
where the application of weak* convergence to the set $D_0$ is justified since it has negligible boundary with respect to the limit measure $\nu_{x_0}$. 
In view of Lemma~\ref{lem:power}, Theorem~\ref{thm:BQ} also applies to the $d$-step random walk given by $\mu^{*d}$. 
Assuming for the moment that the limit measure for this $d$-step random walk starting from $x_0$ coincides with $\nu_{x_0}$, we deduce that 
\begin{align}\label{second_value}
\nu_{x_0}(D_0)=\lim_{n\to\infty}\frac1n\sum_{k=0}^{n-1}(\mu^{*dk}*\delta_{x_0})(D_0)=1.
\end{align}
Together, \eqref{first_value} and \eqref{second_value} imply $d=1$, the desired conclusion. 

It thus remains to show that the $d$-step random walk starting from $x_0$ does indeed have the same limit measure as the $1$-step random walk. 
Denoting by $\calS$ and $\calS_d$ the closed subsemigroups of $G$ generated by $\supp(\mu)$ and $\supp(\mu^{*d})$, respectively, this statement is equivalent to the equality $\overline{\calS x_0}=\overline{\calS_d x_0}$ of orbit closures. 
To prove this, let $g\in\supp(\mu)$ be arbitrary. 
We claim that 
\begin{align*}
\overline{\calS x_0}=\bigcup_{k=0}^{d-1}g^{-k}\overline{\calS_d x_0}.
\end{align*}
Indeed, since $\overline{\calS x_0}$ is homogeneous, it is invariant under the group generated by $\calS$. 
As $\overline{\calS x_0}$ clearly contains $\overline{\calS_d x_0}$, the inclusion \enquote{$\supset$} follows. 
For the reverse inclusion let $g_n,\dots,g_1\in\supp(\mu)$ for some $n\in \N$. 
Choose $0\le k<d$ such that $n+k\equiv 0\bmod d$. 
Then $g^kg_n\dotsm g_1x_0\in\overline{\calS_d x_0}$ and hence $g_n\dotsm g_1 x_0\in g^{-k}\overline{\calS_d x_0}$, giving the claim. 

We already noted that Theorem~\ref{thm:BQ} applies to $\mu^{*d}$. 
In particular, the orbit closure $\overline{\calS_d x_0}$ and its translates by $g^{-k}$, $0\le k<d$, are submanifolds of $\overline{\calS x_0}$. 
Necessarily, all these translates have the same dimension, and since together they make up $\overline{\calS x_0}$ by the claim above, their shared dimension coincides with that of $\overline{\calS x_0}$. 
This implies that $\overline{\calS_d x_0}$ is open in $\overline{\calS x_0}$. 
However, it is also closed, so that the assumed connectedness of $\overline{\calS x_0}$ forces $\overline{\calS x_0}=\overline{\calS_d x_0}$. 
This completes the proof. 
\end{proof}
\begin{remark}\label{rmk:first_moment}
It was recently shown by B\'enard--de Saxc\'e~\cite{benard-desaxce} that the compact support assumption on $\mu$ in Theorem~\ref{thm:BQ} can be relaxed. Indeed, their \cite[Theorem~C]{benard-desaxce} establishes the same conclusion under the substantially weaker assumption that $\mu$ has a \emph{finite first moment}, meaning that
\begin{align*}
\int_G\log\norm{\Ad(g)}\dd \mu(g)<\infty.
\end{align*}
Relying on this stronger result, also our Theorem~\ref{thm:aperiodicity} above and Corollary~\ref{cor:arithmetic_cesaro}  below are seen to hold under a finite first moment assumption on $\mu$, instead of requiring compact support as in Theorem~\ref{thm:BQ}.
\end{remark}
We end this section by recording a corollary of the proof above. 
\begin{corollary}\label{cor:arithmetic_cesaro}
Retain the notation and assumptions from Theorem~\textup{\ref{thm:BQ}} and suppose that $\overline{\calS x_0}$ is connected. 
Let $d\in\N$ and denote by $\calS_d$ the closed subsemigroup of $G$ generated by $\supp(\mu^{*d})$. 
Then $\overline{\calS x_0}=\overline{\calS_d x_0}$, and for the homogeneous probability measure $\nu_{x_0}$ on this orbit closure we have for arbitrary $r\in\N_0$ that 
\begin{align}\label{arithmetic}
\frac1n\sum_{k=0}^{n-1}\mu^{*(dk+r)}*\delta_{x_0}\longrightarrow \nu_{x_0}
\end{align}
as $n\to\infty$ in the weak* topology. 
\end{corollary}
\begin{proof}
The statement about orbit closures was established as part of the proof of Theorem~\ref{thm:aperiodicity}. 
From Theorem~\ref{thm:BQ} we thus get the weak* convergence 
\begin{align}\label{r_equal_0}
\frac1n\sum_{k=0}^{n-1}\mu^{*dk}*\delta_{x_0}\overset{n\to\infty}{\longrightarrow} \nu_{x_0},
\end{align}
which is \eqref{arithmetic} for $r=0$. 
Given $f\in C_c(X)$, the general case follows by applying \eqref{r_equal_0} to the compactly supported continuous function $f_r$ defined by 
\begin{align*}
f_r(x)\coloneqq\int_Gf(gx)\dd\mu^{*r}(g)=\int_{G^r}f(g_r\dotsm g_1x)\dd\mu^{\otimes r}(g_1,\dots,g_r)
\end{align*}
for $x\in X$. 
\end{proof}
This corollary sharpens the convergence statement in Theorem~\ref{thm:BQ} in the case of a connected orbit closure: 
The Ces\`aro convergence to $\nu_{x_0}$ holds along arbitrary arithmetic progressions. 
Although this does not provide an answer to Question~\ref{q}, it at least allows the following conclusion to be drawn: 
If $(n_i)_i$ is a sequence of indices such that $\mu^{*n_i}*\delta_{x_0}$ converges to a weak* limit different from $\nu_{x_0}$ as $i\to\infty$, then $(n_i)_i$ cannot contain a density $1$ subset of an infinite arithmetic progression. 
\section{Spectral Gap}\label{sec:spectral_gap}
In this section, we will explain how a spectral gap of the convolution operator $\pi(\mu)$ associated to a random walk entails the convergence of $\mu^{*n}*\delta_x$ towards $m_X$ for $m_X$-a.e.\ $x\in X$. 
In its simplest form, the involved argument works in great generality and also produces an exponential rate of convergence from almost every starting point when the test function $f$ is fixed. 
This is done in \S\ref{subsec:gen_pts}. 
The following subsections \S\ref{subsec:hght_fct}--\S\ref{subsec:exp_gen} are dedicated to a substantial refinement of this spectral gap argument for random walks on homogeneous spaces of real Lie groups, making the exponentially fast convergence uniform over smooth test functions. 

\subsection{Generic Points}\label{subsec:gen_pts}
Recall that $\pi(\mu)\colon L^\infty(X,m_X)\to L^\infty(X,m_X)$ is defined by 
\begin{align*}
\pi(\mu)f(x)\coloneqq\int_Xf\dd(\mu*\delta_x)=\int_Gf(gx)\dd\mu(g)
\end{align*}
for $f\in L^\infty(X,m_X)$ and $x\in X$, and that it extends to a continuous contraction on each $L^p$-space (see \cite[Corollary~2.2]{BQ_book}). 
We shall study its behavior on $L^2(X,m_X)$. 
By ergodicity, the $G$-fixed functions are the constant functions, so we restrict our attention to their orthogonal complement $L_0^2(X,m_X)$ of $L^2$-functions with mean $0$. 
\begin{definition}
We say that $\mu$ has a \emph{spectral gap} on $X$ if the associated convolution operator $\pi(\mu)$ restricted to $L_0^2(X,m_X)$ has spectral radius strictly less than $1$. 
\end{definition}
We are going to use the notation $\rho(T)$ to denote the spectral radius of an operator~$T$. Then by the spectral radius formula, $\mu$ having a spectral gap on $X$ can be reformulated as the requirement that 
\begin{align*}
\rho\bigl(\pi(\mu)|_{L_0^2}\bigr)=\lim_{n\to\infty}\sqrt[n]{\norm{\pi(\mu)|^n_{L_0^2}}_{\mathrm{op}}}<1.
\end{align*}

Given the existence of a spectral gap, we obtain an almost everywhere convergence result in a quite general setup. 
\begin{theorem}\label{thm:spectral_gap}
Suppose that $\mu$ has a spectral gap on $X$. 
Then $m_X$-a.e.\ $x\in X$ is \emph{generic} for the random walk on $X$ given by $\mu$, meaning that 
\begin{align*}
\mu^{*n}*\delta_x\longrightarrow m_X
\end{align*}
as $n\to\infty$ in the weak* topology. 
This convergence is exponentially fast in the sense that for every fixed $f\in L^2(X,m_X)$ we have 
\begin{align}\label{desired_statement}
\limsup_{n\to\infty}\abs*{\int_Xf\dd(\mu^{*n}*\delta_x)-\int f\dd m_X}^{1/n}\le \rho\bigl(\pi(\mu)|_{L_0^2}\bigr)
\end{align}
for $m_X$-a.e.\ $x\in X$. 
\end{theorem}
\begin{proof}
By separability of $C_c(X)$, for the statement about weak* convergence it suffices to prove $m_X$-a.s.\ convergence for one fixed function $f\in C_c(X)$. 
Consequently, it is enough to prove the second assertion of the theorem. 
To this end, fix a function $f\in L^2(X,m_X)$ and a rational number $\rho\bigl(\pi(\mu)|_{L_0^2}\bigr)<\alpha<1$, and consider the $L_0^2$-function $f_0=f-\int f\dd m_X$. 
Then in view of the spectral radius formula we have 
\begin{align*}
\norm[\Big]{\pi(\mu)^nf-\int f\dd m_X}_{L^2}=\norm{\pi(\mu)^nf_0}_{L^2}\le \norm{\pi(\mu)|^n_{L_0^2}}_{\mathrm{op}}\norm{f_0}_{L^2}\le \alpha^n\norm{f_0}_{L^2}
\end{align*}
for sufficiently large $n\in\N$. 

Fix in addition a rational number $\epsilon\in(0,1)$. By Chebyshev's inequality, the above implies that for large $n$ we have 
\begin{gather*}
m_X\br*{\set*{x\in X\for \abs[\Big]{\pi(\mu)^nf(x)-\int f\dd m_X}\ge \alpha^{n(1-\epsilon)}\norm{f_0}_{L^2}}}\\\le \frac{\norm*{\pi(\mu)^nf-\int f\dd m_X}_{L^2}^2}{\alpha^{2n(1-\epsilon)}\norm{f_0}_{L^2}^2}\le\alpha^{2\epsilon n}.
\end{gather*}
By Borel--Cantelli it follows that for all $x$ in a full measure set $A_{\alpha,\epsilon}$, the inequality 
\begin{align*}
\abs[\Big]{\pi(\mu)^nf(x)-\int f\dd m_X}\ge \alpha^{n(1-\epsilon)}\norm{f_0}_{L^2}
\end{align*}
holds only for finitely many $n\in\N$. 
Since $\pi(\mu)^nf(x)=\int f\dd(\mu^{*n}*\delta_x)$, we conclude that \eqref{desired_statement} holds for all $x$ in a countable intersection of the sets $A_{\alpha,\epsilon}$ over rational numbers $\alpha$ approaching $\rho\bigl(\pi(\mu)|_{L_0^2}\bigr)$ and $\epsilon$ approaching $0$ from above. 
\end{proof}
\begin{remark}\label{rmk:quantitative}
In the second conclusion of Theorem~\ref{thm:spectral_gap}, how long it takes for the exponential rate of convergence to kick in depends on the point $x$. 
However, the measure of sets on which one has to wait for a long time can be controlled as follows: 
Given $\rho\bigl(\pi(\mu)|_{L_0^2}\bigr)<\alpha<1$, choose $N\in\N$ such that $\norm{\pi(\mu)|^n_{L_0^2}}_{\mathrm{op}}\le \alpha^n$ for all $n\ge N$. 
Then if we additionally take $\epsilon\in(0,1)$ and denote 
\begin{align*}
B_{\alpha,\epsilon,n,f}=\set*{x\in X\for \abs[\Big]{\pi(\mu)^{n'}f(x)-\int f\dd m_X}\ge \alpha^{n'(1-\epsilon)}\norm{f_0}_{L^2}\text{ for some }n'\ge n},
\end{align*}
the proof above gives the bound 
\begin{align*}
m_X\br*{B_{\alpha,\epsilon,n,f}}\le \frac{\alpha^{2\epsilon n}}{1-\alpha^{2\epsilon}}
\end{align*}
for every $n\ge N$. 
In particular, the measure of the set on which the exponential convergence does not start during the first $n$ steps decays exponentially in $n$. 
\end{remark}
We now demonstrate that the previous result covers the case announced in~\S\ref{sec:intro}. 
\begin{proposition}\label{prop:spectral_gap_example}
Let $G$ be a connected semisimple real Lie group without compact factors and with finite center, $\Gamma\subset G$ a lattice, and $X$ the homogeneous space $G/\Gamma$ endowed with the Haar measure $m_X$. 
Suppose that the closed subsemigroup $\calS$ generated by $\supp(\mu)$ has the property that $\Ad(\calS)$ is Zariski dense in $\Ad(G)$. 
Then $\mu$ has a spectral gap on $X$. 
\end{proposition}
\begin{proof}
Consider the regular representation of $G$ on $L_0^2(X,m_X)$. 
By \cite[Lemma~3]{bekka} it doesn't weakly contain the trivial representation. 
From this, in view of \cite[Theorem~C]{spectral_gap}, the result follows if we can argue that the projection of $\mu$ to any simple factor of $G$ is not supported on a closed amenable subgroup. 
However, since amenability passes to the Zariski closure (see e.g.\ \cite[Theorem~4.1.15]{Z}) the latter would imply that one of the simple factors of $\Ad(G)$ is amenable, hence compact by a classical result of Furstenberg (see e.g.\ \cite[Proposition~4.1.8]{Z}). 
\end{proof}
\subsection{Good Height Functions}\label{subsec:hght_fct}
Inspecting the proof of Theorem~\ref{thm:spectral_gap}, one observes that every step is effective, with explicit bounds and good control over the measure of exceptional sets, except for the very first one: separability of the space $C_c(X)$ of compactly supported continuous functions. 
In the remainder of this section, we aim to also make effective this step, the goal being exponentially fast convergence $\mu^{*n}*\delta_x\rightarrow m_X$ from almost every starting point, uniformly over functions $f$ on $X$. 
As merely continuous functions can behave arbitrarily badly (with respect to the convergence problem at hand), there is no hope of achieving this feat for all $f\in C_c(X)$. 
We shall therefore restrict our attention to smooth functions of compact support, and take into account their regularity by considering not just their $L^2$, but also certain Sobolev norms. 
Built into the definition of these norms will be what we call a \emph{good height function}, the concept of which is introduced in this subsection. 

Our setup is as follows: 
Let $G$ be a real Lie group with Lie algebra $\g$. 
We endow $\g$ with a scalar product, which we use to define a right-invariant metric $\dist^G$ on $G$. 
Given a lattice $\Gamma\subset G$, this metric descends to a metric $\dist^X$ on $X=G/\Gamma$ such that the projection $G\to X$ is locally an isometry. 
Moreover, we fix an orthonormal basis of $\g$, using which we will identify $\g$ with $\R^{\dim\g}$. 
Here is the crucial definition. 
\begin{definition}
We call a measurable function $\hght\colon X\to(0,\infty)$ a \emph{good height function} if there exists $0<R\le 1$ and a function $r\colon X\to(0,R]$ with the following properties: 
\begin{enumerate}
\item The restriction of the exponential map $\exp\colon (-R,R)^{\dim\g}\to G$ is a diffeomorphism onto its image and we have $\exp((-r/2,r/2)^{\dim\g})\subset B_r^G(e)$ for all $r\le R$, where $B_r^G(e)$ denotes the open ball of radius $r$ around the identity $e\in G$ with respect to the metric $\dist^G$ on $G$. 
\item For all $x\in X$, the projection $G\supset B_{r(x)}^G(e)\to X,\,g\mapsto gx$ is injective. 
\item There exist constants $c,\kappa>0$ such that $r(x)\ge c\hght(x)^{-\kappa}$ for all $x\in X$. 
\item There exists a constant $\sigma>1$ such that $\hght(x)\le \sigma\hght(gx)$ for all $x\in X$ and all $g\in B_{r(x)}^G(e)$. 
\end{enumerate}
\end{definition}

The definition suggests to think of a good height function as reciprocal of the injectivity radius. 
And indeed, this viewpoint allows their construction on any homogeneous space $X=G/\Gamma$. 
\begin{proposition}\label{prop:good_hght_existence}
Let $G$ be a real Lie group and $\Gamma$ a lattice in $G$. 
Then $X=G/\Gamma$ admits a good height function. 
\end{proposition}
\begin{proof}
Choose $R>0$ such that condition (i) of the definition is satisfied and set $r(x)=\min\set{R,r_{\mathrm{inj}}(x)}$, where $r_{\mathrm{inj}}(x)$ is the injectivity radius at $x\in X$, i.e.\ the maximal radius such that (ii) holds at $x$. 
Define 
\begin{align*}
\hght(x)=r(x)^{-1}.
\end{align*}
Then the only thing that needs to be verified is the validity of (iv). 
We claim that it holds with $\sigma=2$. 
This will follow if we can show that 
\begin{align}\label{radius_comparison}
r_{\mathrm{inj}}(gx)\le 2r_{\mathrm{inj}}(x)
\end{align}
whenever $g\in B_{r(x)}^G(e)$. 
To this end, let $r>r_{\mathrm{inj}}(x)$. 
Then by definition, there are distinct $g_1,g_2\in B_r^G(e)$ such that $g_1x=g_2x$. 
As $g\in B_{r(x)}^G(e)$, right-invariance of the metric implies 
\begin{align*}
\dist^G(g_ig^{-1},e)=\dist^G(g_i,g)\le \dist^G(g_i,e)+\dist^G(g,e)<r+r(x)<2r
\end{align*}
for $i=1,2$, and we also have $(g_1g^{-1})gx=(g_2g^{-1})gx$. 
This shows that $r_{\mathrm{inj}}(gx)\le 2r$, and as $r>r_{\mathrm{inj}}(x)$ was arbitrary, we see that \eqref{radius_comparison} holds. 
\end{proof}
Often, however, one might want to work with different, naturally occurring height functions. 
The flexibility in our definition of a good height function accommodates this possibility. 

In the examples below, we denote by $\lambda_1(\Lambda)$ the length of a shortest non-zero vector in a lattice $\Lambda\subset\R^d$. 
\begin{example}\label{ex:lattice_height_fct}
Let $G=\SL_d(\R)$ and $\Gamma=\SL_d(\Z)$. 
Then $X=G/\Gamma$ can be identified with the space of lattices in $\R^d$ with covolume 1 via 
\begin{align*}
X\ni g\SL_d(\Z)\longleftrightarrow g\Z^d\subset\R^d.
\end{align*}
Then the function $\hght=\lambda_1^{-1}$, defined on $X$ via the above identification, is a good height function. 
Indeed, one can first choose $R>0$ such that (i) is satisfied, and then set $r(x)=\min\set{R,r_{\mathrm{inj}}(x)}$ as in the proof of Proposition~\ref{prop:good_hght_existence}. 
Then (ii) is automatically satisfied, and (iv) is valid for a suitable choice of $\sigma$ due to the inequality $\lambda_1(gx)\le \norm{g}\lambda_1(x)$ for $g\in G$ and $x\in X$, where $\norm{\cdot}$ denotes any matrix norm. 
To see that also (iii) holds, let $x=g\Gamma$ and suppose that $hx=x$ for some $h\in G$ with $h\neq e$. 
Then for all $\gamma\in\SL_d(\Z)$, the matrix $(g\gamma)^{-1}h(g\gamma)$ fixes the lattice $\Z^d$ but is not the identity, so that 
\begin{align*}
\norm{g\gamma}^{\kappa_1}\norm{h-e}\ge\norm{(g\gamma)^{-1}(h-e)(g\gamma)}=\norm{(g\gamma)^{-1}h(g\gamma)-e} \ge c_1
\end{align*}
for some constants $c_1,\kappa_1>0$. 
For a basis change $\gamma\in\SL_d(\Z)$ such that $g\gamma$ consists of a reduced basis of the lattice $x$ we have $\norm{g\gamma}\le c_2\lambda_1(x)^{-\kappa_2}$ for some $c_2,\kappa_2>0$ (cf.\ e.g.\ \cite[Chapter~III]{geometry}). 
With this choice, the above inequality implies 
\begin{align*}
\norm{h-e}\ge c\lambda_1(x)^\kappa
\end{align*}
for $c=c_1/c_2$ and $\kappa=\kappa_1\kappa_2$. 
Since near the identity, the metric $\dist^G$ on $G$ is Lipschitz-equivalent to the distance induced by $\norm{\cdot}$, this establishes (iii). 
\end{example}
A similar construction is possible in a more general context. 
\begin{example}[\cite{EMV}]\label{ex:EMV}
Let $G=\mathbb{G}(\R)$ be the group of real points of a semisimple $\Q$-group $\mathbb{G}$ and $\Gamma$ an arithmetic lattice in $G$. 
Choose a rational $\Ad(\Gamma)$-stable lattice $\g_\Z\subset\g$. 
Then, using similar reasoning as in the previous example, the function $\hght$ on $X=G/\Gamma$ defined by 
\begin{align*}
\hght(x)=\lambda_1(\Ad(g)\g_\Z)^{-1}
\end{align*}
for $x=g\Gamma\in X$ is seen to be a good height function (cf.\ \cite[\S3.6]{EMV}). 
%
\end{example}
\subsection{Sobolev Norms}\label{subsec:sobolev}
Given a good height function $\hght$ on $X$, the associated \emph{Sobolev norm} of degree $\ell\ge 0$ of a compactly supported smooth function $f\in C_c^\infty(X)$ is defined by 
\begin{align*}
\calS_\ell(f)^2=\sum_{\deg\mathcal{D}\le \ell}\norm{\hght(\cdot)^\ell\mathcal{D}f}^2_{L^2},
\end{align*}
where the sum runs over differential operators $\mathcal{D}$ given by monomials of degree at most $\ell$ in elements of the fixed orthonormal basis of $\g$ in the universal enveloping algebra. 

In other words, the differential operators $\mathcal{D}$ appearing above are $\partial_{v_1}\dotsm\partial_{v_k}$ for any $k$-tuple $(v_1,\dots,v_k)$ of elements of the fixed basis of $\g$, $0\le k\le \ell$, where $\partial_v$ for $v\in\g$ is defined by 
\begin{align*}
\partial_vf(x)=\lim_{t\to 0}\frac{f(\exp(tv)x)-f(x)}{t}
\end{align*}
for $f\in C_c^\infty(X)$ and $x\in X$. 

Here are two immediate observations. 
\begin{lemma}\label{lem:sobolev_increasing}
Let $\hght$ be a good height function on $X$ and $\calS_\ell$ the associated Sobolev norms. 
\begin{enumerate}
\item The norms $\calS_\ell$ are induced by inner products $\langle\cdot,\cdot\rangle_\ell$ on $C_c^\infty(X)$. 
\item Given $0\le\ell_0\le\ell_1$, there exists a constant $\tilde{c}>0$ such that $\calS_{\ell_0}\le\tilde{c}\calS_{\ell_1}$. 
\end{enumerate}
\end{lemma}
\begin{proof}
Part (i) is clear. 
Part (ii) is also immediate from the definition of the Sobolev norms, once we know that a good height function must be bounded away from $0$. 
The latter, however, follows directly from property (iii) in the definition of a good height function, as the function $r$ appearing there is assumed to be bounded. 
\end{proof}

The proof of our convergence result in \S\ref{subsec:exp_gen} will depend on the following proposition. 
\begin{proposition}[\cite{EMV}]\label{prop:sobolev_facts}
For the Sobolev norms associated to a good height function on $X$, there exists a non-negative integer $\ell_0\ge 0$ and a constant $C>0$ with the following properties: 
\begin{enumerate}
\item \textup{(}Sobolev embedding estimate, \cite[(3.9)]{EMV}\textup{)} For every $f\in C_c^\infty(X)$ it holds that $\norm{f}_\infty\le C\calS_{\ell_0}(f)$. 
\item \textup{(}Finite relative traces, \cite[(3.10)]{EMV}\textup{)} For all integers $\ell\ge 0$ the relative trace $\operatorname{Tr}(\calS_\ell^2|\calS_{\ell+\ell_0}^2)$ is finite, meaning that for any orthogonal basis $(e^{(k)})_k$ in the completion of $C_c^\infty(X)$ with respect to $\calS_{\ell+\ell_0}$ 
\begin{align*}
\operatorname{Tr}(\calS_\ell^2|\calS_{\ell+\ell_0}^2)\coloneqq\sum_k\frac{\calS_\ell(e^{(k)})^2}{\calS_{\ell+\ell_0}(e^{(k)})^2}<\infty.
\end{align*}
\end{enumerate}
\end{proposition}
We refer to Bernstein--Reznikov~\cite{rel_trace} for a systematic treatment of relative traces. 
In particular, it is proved in this reference that the above expression is independent of the choice of orthogonal basis. 

The proofs in~\cite{EMV} of the statements in the above proposition are given for the height function from Example~\ref{ex:EMV}. 
However, the only properties used are those in our definition of a good height function. 
In fact, the arguments only depend on validity of the second statement in \cite[Lemma~5.1]{EMV}, which holds in our context, as we demonstrate below. 
\begin{lemma}
Let $\hght$ be a good height function on $X$. 
Then there exists a non-negative integer $\ell_0\ge 0$ and a constant $C>0$ such that for every non-negative integer $\ell\ge 0$ and every differential operator $\mathcal{D}$ given by a monomial of degree at most $\ell$ in elements of the fixed basis of $\g$ we have 
\begin{align*}
\abs{\hght(x)^\ell\mathcal{D}f(x)}\le C\calS_{\ell+\ell_0}(f)
\end{align*}
for every $f\in C_c^\infty(X)$ and $x\in X$. 
\end{lemma}
\begin{proof}
We inspect the function $F=\mathcal{D}f$ in a chart around $x$ given by the exponential map: 
We set $\epsilon=r(x)/2$, where $r\colon X\to(0,R]$ is the function from the definition of a good height function, $d=\dim\g$, and consider 
\begin{align*}
\tilde{F}\colon (-\epsilon,\epsilon)^d\to \R,\, v\mapsto F(\exp(v)x).
\end{align*}
Then by the first statement of \cite[Lemma~5.1]{EMV}, which is simply a Sobolev embedding estimate on $\R^d$, we know 
\begin{align}\label{sobolev_in_chart}
\abs{F(x)}=\abs{\tilde{F}(0)}\le C_12^dr(x)^{-d}\calS_{d,\epsilon}(\tilde{F}),
\end{align}
where $C_1>0$ is a constant depending only on the dimension $d$ of $\g$ and $\calS_{d,\epsilon}$ is the standard degree $d$ Sobolev norm on the open subset $(-\epsilon,\epsilon)^d$ of $\R^d$, i.e.\
\begin{align*}
\calS_{d,\epsilon}(\tilde{F})^2=\sum_{\abs{\boldsymbol{\alpha}}\le d}\norm{\partial_{\boldsymbol{\alpha}}\tilde{F}}^2_{L^2((-\epsilon,\epsilon)^d)},
\end{align*}
where the sum is over all multi-indices $\boldsymbol{\alpha}$ of degree at most $d$ and $\partial_{\boldsymbol{\alpha}}\tilde{F}$ is the corresponding standard partial derivative of $\tilde{F}$. 
Using property (iii) in the definition of a good height function, \eqref{sobolev_in_chart} implies that 
\begin{align}\label{estimate_with_hght}
\abs{\hght(x)^\ell F(x)}\le C_2\hght(x)^{\ell+\ell_0}\calS_{d,\epsilon}(\tilde{F}),
\end{align}
where $C_2>0$ is another constant and we used that $\hght$ is bounded away from $0$ to replace $\kappa d$ appearing in the exponent by $\ell_0=\max\{\lceil\kappa d\rceil,d\}$. 
Using properties (i) and (ii) in the definition of a good height function, we find $C_3>0$ such that 
\begin{align}\label{comparison_inequality}
\calS_{d,\epsilon}(\tilde{F})\le C_3\sqrt{\sum_{\deg\mathcal{D}'\le d}\norm{\mathcal{D}'F|_{B_{r(x)}^X(x)}}^2_{L^2}}.
\end{align}
To see this, one needs to note two things: firstly, that by the chain rule the partial derivatives of $\tilde{F}$ at a point $v\in (-\epsilon,\epsilon)^d$ in the chart can be expressed as linear combinations of derivatives $\mathcal{D}'F$ appearing on the right-hand side in \eqref{comparison_inequality} evaluated at the corresponding point $x'=\exp(v)x$, with fixed coefficient functions depending only on finitely many derivatives of the exponential map on $(-\epsilon,\epsilon)^d$; and secondly, that the Haar measure $m_X$ is a smooth measure, meaning that it has a smooth and nowhere vanishing density w.r.t.\ Lebesgue measure in the chart. 

Combining \eqref{estimate_with_hght}, \eqref{comparison_inequality}, condition (iv) in the definition of a good height function, and plugging back in the definition of $F$, we finally arrive at 
\begin{align*}
\abs{\hght(x)^\ell\mathcal{D}f(x)}\le C_4\sqrt{\sum_{\deg\mathcal{D}'\le d}\norm{\hght(\cdot)^{\ell+\ell_0}\mathcal{D}'\mathcal{D}f|_{B_{r(x)}^X(x)}}^2_{L^2}}\le C_4\calS_{\ell+\ell_0}(f),
\end{align*}
for yet another constant $C_4>0$, which is the one appearing in the lemma. 
\end{proof}
\subsection{Exponentially Generic Points}\label{subsec:exp_gen}
Now we are ready to define the notion of effective genericity we wish to establish, and to prove the main convergence result of this section. 

Until the end of this section, we fix a good height function $\hght$ on $X$. 
Moreover, given a bounded measurable function $f$ on $X$ and $n\in\N$ we will use the notation
\begin{align*}
D_n(f)(x)=\pi(\mu)^nf(x)-\int f\dd m_X
\end{align*}
for $x\in X$. 
We refer to $D_n(f)$ as the \emph{time $n$ discrepancy} for the function $f$.  
\begin{definition}\label{def:exp_generic}
We say that a point $x\in X$ is \emph{$(\ell,\beta)$-exponentially generic} if $\ell\ge 0$ is a non-negative integer and $\beta$ a real number in $(0,1)$ satisfying 
\begin{align*}
\limsup_{n\to\infty}\sup_{f\in C_c^\infty(X)\setminus\set{0}}\biggl(\frac{\abs{D_n(f)(x)}}{\calS_\ell(f)}\biggr)^{1/n}\le \beta,
\end{align*}
where $\calS_\ell$ is the degree $\ell$ Sobolev norm associated to $\hght$. 
\end{definition}
With this terminology, we have the following result, which quantifies the dependence on the function $f$ in the effective part of Theorem~\ref{thm:spectral_gap}. 
\begin{theorem}\label{thm:exp_generic}
Let $G$ be a real Lie group, $\Gamma\subset G$ a lattice and $X=G/\Gamma$ endowed with the Haar measure $m_X$. 
Suppose that $\mu$ has a spectral gap on $X$. 
Then there exists a non-negative integer $\ell_1\ge 0$ such that $m_X$-almost every point $x\in X$ is $\bigl(\ell_1,\rho\bigl(\pi(\mu)|_{L_0^2}\bigr)\bigr)$-exponentially generic. 
\end{theorem}
Our argument uses ideas from the proof of \cite[Proposition~9.2]{EMV}. 
Recall that $\langle\cdot,\cdot\rangle_\ell$ denotes the inner product associated to the Sobolev norm $\calS_\ell$. 
\begin{proof}
Set $\ell_1=2\ell_0$ with $\ell_0$ from Proposition~\ref{prop:sobolev_facts}. 
We denote by $\calH$ the completion of $C_c^\infty(X)$ with respect to $\calS_{\ell_1}$. 

The first step of the proof is to argue that $\calH$ admits an orthonormal basis $(e^{(k)})_k$ with respect to $\calS_{\ell_1}$ that is also orthogonal with respect to $\calS_{\ell_0}$. 
To this end, let us endow $\calH$ with the scalar product $\langle\cdot,\cdot\rangle_{\ell_1}$ associated to $\calS_{\ell_1}$. 
This makes $\calH$ into a Hilbert space. 
As a consequence of Lemma~\ref{lem:sobolev_increasing}(ii), $\langle\cdot,\cdot\rangle_{\ell_0}$ defines a bounded positive definite Hermitian form on $(\calH,\langle\cdot,\cdot\rangle_{\ell_1})$. 
Using Riesz representation it follows that there is a bounded positive self-adjoint operator $T$ on $(\calH,\langle\cdot,\cdot\rangle_{\ell_1})$ such that
\begin{align*}
\langle v,w\rangle_{\ell_0}=\langle Tv,w\rangle_{\ell_1}
\end{align*}
for all $v,w\in\calH$. 
Finiteness of the relative trace $\operatorname{Tr}(\calS_{\ell_0}^2|\calS_{\ell_1}^2)$ from Proposition~\ref{prop:sobolev_facts}(ii) then translates into the statement that $T$ is a trace-class operator on $(\calH,\langle\cdot,\cdot\rangle_{\ell_1})$ (cf.\ \cite[Proposition~6.44]{func_ana}); in particular, the operator $T$ is compact (cf.\ \cite[Proposition~6.42]{func_ana}). 
By the spectral theorem, $T$ is thus diagonalizable. 
Hence, an orthonormal basis $(e^{(k)})_k$ of $(\calH,\langle\cdot,\cdot\rangle_{\ell_1})$ consisting of eigenvectors of $T$ is a basis with the desired properties. 

Next, fix rational numbers $\rho\bigl(\pi(\mu)|_{L_0^2}\bigr)<\alpha<1$ and $\epsilon\in(0,1)$. 
As in the proof of Theorem~\ref{thm:spectral_gap}, using Chebyshev's inequality we find that for every $k\ge 0$ and large enough $n$ we have 
\begin{gather}
m_X\br*{\set*{x\in X\for \abs[\big]{D_n(e^{(k)})(x)}\ge \alpha^{n(1-\epsilon)}\calS_{\ell_0}(e^{(k)})}}\nonumber\\\le\frac{\norm{e_0^{(k)}}_{L^2}^2}{\calS_{\ell_0}(e^{(k)})^2}\alpha^{2\varepsilon n}\le\frac{\norm{e^{(k)}}_{L^2}^2}{\calS_{\ell_0}(e^{(k)})^2}\alpha^{2\varepsilon n},\label{RHS_above}
\end{gather}
where $e_0^{(k)}=e^{(k)}-\int e^{(k)}\dd m_X$. 
Since the relative trace $\operatorname{Tr}(\calS_0^2|\calS_{\ell_0}^2)$ is finite by Proposition~\ref{prop:sobolev_facts}, the terms on the right-hand side of \eqref{RHS_above} are summable over $k,n\ge 0$. 
Borel--Cantelli thus implies that 
\begin{align*}
\limsup_{k,n\ge 0}\set*{x\in X\for \abs[\big]{D_n(e^{(k)})(x)}\ge \alpha^{n(1-\epsilon)}\calS_{\ell_0}(e^{(k)})}
\end{align*}
is a null set. Let $A_{\alpha,\epsilon}$ be the complement of this null set. 
We claim that any $x\in A_{\alpha,\epsilon}$ is $(\ell_1,\alpha^{1-\varepsilon})$-exponentially generic. 
Fix such a point $x$. Then we know that there are only finitely many pairs $(k,n)$ with $\abs{D_n(e^{(k)})(x)}\ge\alpha^{n(1-\epsilon)}\calS_{\ell_0}(e^{(k)})$. Thus, there exists $n_0$ such that for $n\ge n_0$ the inequality $\abs{D_n(e^{(k)})(x)}<\alpha^{n(1-\epsilon)}\calS_{\ell_0}(e^{(k)})$ holds for all $k$.  Now let $f\in C_c^\infty(X)\setminus\set{0}$ be arbitrary and write $f=\sum_k f_ke^{(k)}$ for the expansion of $f$ in terms of the orthonormal basis $(e^{(k)})_k$. 
Then, using the triangle inequality, we can estimate the time $n$ discrepancy for $f$ as follows: 
\begin{align}\label{series_convergence}
\abs{D_n(f)(x)}\le \sum_k\abs{f_k}\abs{D_n(e^{(k)})(x)}.
\end{align}
The exchange of integral and summation involved in the above estimate is justified by part (i) of Proposition~\ref{prop:sobolev_facts}: 
It ensures that the functions $e^{(k)}$ are defined pointwise and the series expansion of $f$ converges uniformly. Next, for $n\ge n_0$ an application of the Cauchy--Schwarz inequality implies that the right-hand side of \eqref{series_convergence} is strictly less than 
\begin{align}\label{conclusion}
\alpha^{n(1-\epsilon)}\biggl(\sum_k\abs{f_k}^2\biggr)^{1/2}\biggl(\sum_k\calS_{\ell_0}(e^{(k)})^2\biggr)^{1/2}=\alpha^{n(1-\epsilon)}\calS_{\ell_1}(f)\operatorname{Tr}(\calS_{\ell_0}^2|\calS_{\ell_1}^2)^{1/2}.
\end{align}
Again by Proposition~\ref{prop:sobolev_facts}, the relative trace $\operatorname{Tr}(\calS_{\ell_0}^2|\calS_{\ell_1}^2)$ is finite. 
Hence, in view of our definition of exponential genericity and the fact that $n_0$ does not depend on $f$, combining \eqref{series_convergence} and \eqref{conclusion} establishes the claim. 
It follows that all $x$ in a countable intersection of the sets $A_{\alpha,\epsilon}$ over rational numbers $\alpha$ approaching $\rho\bigl(\pi(\mu)|_{L_0^2}\bigr)$ and $\epsilon$ approaching $0$ from above are $\bigl(\ell_1,\rho\bigl(\pi(\mu)|_{L_0^2}\bigr)\bigr)$-exponentially generic, giving the theorem. 
\end{proof}
\begin{remark}
In analogy to Remark~\ref{rmk:quantitative}, we can control the measure of the set of points where exponentially generic behavior is not observed for a given number of steps: 
If we define
\begin{align*}
B_{\alpha,\epsilon,n}=\Bigl\{x\in X \SetSymbol[\Big] \abs*{D_{n'}(f)(x)}\ge \alpha^{n'(1-\epsilon)}\calS_{\ell_1}(f)&\operatorname{Tr}(\calS_{\ell_0}^2|\calS_{\ell_1}^2)^{1/2}\\&\text{for some }n'\ge n,f\in C_c^\infty(X)\Bigr\}
\end{align*}
for $\rho\bigl(\pi(\mu)|_{L_0^2}\bigr)<\alpha<1$, $\epsilon\in(0,1)$ and $n\in\N$, and $N\in\N$ is chosen such that $\norm{\pi(\mu)|^n_{L_0^2}}_{\mathrm{op}}\le \alpha^n$ for all $n\ge N$, then for every $n\ge N$ it holds that 
\begin{align*}
m_X\br*{B_{\alpha,\epsilon,n}}\le \operatorname{Tr}(\calS_0^2|\calS_{\ell_0}^2)\frac{\alpha^{2\epsilon n}}{1-\alpha^{2\epsilon}}.
\end{align*}
Indeed, we have $B_{\alpha,\epsilon,n}\subset\bigcup_{n'\ge n,k\ge 0}\set[\big]{x\in X\for \abs[\big]{D_{n'}(e^{(k)})(x)}\ge \alpha^{n'(1-\epsilon)}\calS_{\ell_0}(e^{(k)})}$, as the proof of Theorem~\ref{thm:exp_generic} demonstrates. 
Thus, again, the measure of the set of \enquote{bad points}, on which exponential genericity takes more than $n$ steps to manifest, is itself exponentially small in $n$. 
\end{remark}
\section{Uniform Ces\`aro Convergence}\label{sec:uniform}
In this last section, we explore the situation where the only possible limit in Theorem~\ref{thm:BQ} is the normalized Haar measure $m_X$. 
In this setting, by analogy with the case of unique ergodicity in classical ergodic theory, it is reasonable to expect the Ces\`aro convergence \eqref{cesaro} to hold (locally) uniformly in the starting point $x_0$. 
We shall prove in \S\ref{subsec:locally_uniform} below that this indeed holds true. 
In \S\ref{subsec:lyapunov}, we conclude the article by showing that in many naturally occurring situations something even stronger than locally uniform can be achieved. 

Before continuing with the pertinent definitions, let us recall that even though the setup of Theorem~\ref{thm:BQ} is our motivation and useful to have in mind, formally we are working with the assumptions stated at the end of \S\ref{sec:intro}: $(X,m_X)$ is merely required to be a space with a $G$-action for which $m_X$ is invariant and ergodic. 

\begin{definition}
A probability measure $\nu$ on $X$ is called \emph{$\mu$-stationary} if $\mu*\nu=\nu$. 
The random walk on $X$ induced by $\mu$ is called \emph{uniquely ergodic} if $m_X$ is the unique $\mu$-stationary probability measure on $X$. 
\end{definition}
In particular, for a random walk to be uniquely ergodic, there must be no finite $\G$-orbits in $X$, where $\G$ denotes the closed subgroup of $G$ generated by $\mu$. 
In the case that $X=G/\Gamma$ for a lattice $\Gamma$ in $G$, this happens if and only if $\G$ is not virtually contained in a conjugate of $\Gamma$. 
(Recall that a subgroup $H$ of $G$ is said to be \emph{virtually contained} in a subgroup $L$ of $G$ if $H\cap L$ has finite index in $H$.) 
In fact, in many cases of interest, finite orbits are the only obstruction to unique ergodicity: 
For example, this is true when $G$ is a connected semisimple Lie group without compact factors, $\Gamma$ is an irreducible lattice, $X=G/\Gamma$, and $\Ad(\calS)$ is Zariski dense in $\Ad(G)$ (see \cite[Corollary~1.8]{BQ3}); and also in the setting of~\cite{SW}, a special case of which is reproduced below as Example~\ref{ex:SW}. 

\subsection{Locally Uniform Convergence}\label{subsec:locally_uniform}
The notion of unique ergodicity introduced above coincides with the classical property of unique ergodicity of the Markov operator $\pi(\mu)$. 
When the space $X$ is compact, this is enough to guarantee that the Ces\`aro convergence $\frac1n\sum_{k=0}^{n-1}\mu^{*k}*\delta_x\to m_X$ as $n\to\infty$ is uniform in $x$ (see e.g.\ \cite[\S5.1]{krengel}). 
Without compactness, we also need to assume a form of recurrence. 
\begin{definition}
We say that the random walk on $X$ given by $\mu$ is \emph{locally uniformly recurrent} if for every compact subset $K\subset X$ and $\epsilon>0$ there exists $n_0\in\N$ and a compact subset $M\subset X$ with 
\begin{align*}
\mu^{*n}*\delta_x(M)\ge 1-\epsilon
\end{align*}
for all $n\ge n_0$ and $x\in K$. 
It is called \emph{locally uniformly recurrent on average} if the above holds with the Ces\`aro averages $\frac{1}{n}\sum_{k=0}^{n-1}\mu^{*k}*\delta_x$ in place of $\mu^{*n}*\delta_x$. 
\end{definition}
It is a simple exercise to check that locally uniform recurrence implies locally uniform recurrence on average. 
In concrete examples, recurrence properties such as these are typically established by constructing a Lyapunov function; see~\S\ref{subsec:lyapunov} below. 

The following well-known fact explains why these properties are referred to as \enquote{non-escape of mass}. 
\begin{lemma}\label{lem:non_escape}
Let the sequence $\set{x_n}_n$ of points in $X$ be relatively compact and suppose that the random walk on $X$ is locally uniformly recurrent \textup{(}resp.\ on average\textup{)}. 
Then every weak* limit of the sequence $(\mu^{*n}*\delta_{x_n})_n$ \textup{(}resp.\ $\br[\big]{\frac{1}{n}\sum_{k=0}^{n-1}\mu^{*k}*\delta_{x_n}}_n$\textup{)} is a probability measure. \qed
\end{lemma}
The proof is immediate and left to the reader. 

We are now ready to state and prove our first result on locally uniform Ces\`aro convergence. 
\begin{theorem}\label{thm:uniform}
Suppose that the random walk on $X$ induced by $\mu$ is uniquely ergodic and locally uniformly recurrent on average. 
Then for every $f\in C_c(X)$, every compact $K\subset X$, and every $\epsilon>0$, there exists $n_0\in\N$ such that for every $n\ge n_0$ and $x\in K$ we have 
\begin{align*}
\abs[\bigg]{\frac1n\sum_{k=0}^{n-1}\int_Xf\dd(\mu^{*k}*\delta_x)-\int_Xf\dd m_X}<\epsilon.
\end{align*}
Equivalently, considering the space of probability measures on $X$ as endowed with the weak* topology, the sequence of functions 
\begin{align*}
X\ni x\mapsto \frac1n\sum_{k=0}^{n-1}\mu^{*k}*\delta_x
\end{align*}
converges to $m_X$ uniformly on compact subsets of $X$ as $n\to \infty$. 
\end{theorem}
\begin{proof}
The equivalence of the two formulations is due to the definition of neighborhoods in the weak* topology by finitely many test functions in $C_c(X)$. 

To prove the statement for individual functions, we proceed by contradiction. 
If the conclusion is false, then for some $f\in C_c(X)$, $K\subset X$ compact and $\epsilon>0$ there exist indices $n(j)\to\infty$ and $x_j\in K$ with 
\begin{align}\label{wrong_limit}
\abs[\bigg]{\frac{1}{n(j)}\sum_{k=0}^{n(j)-1}\int_Xf\dd(\mu^{*k}*\delta_{x_j})-\int_Xf\dd m_X}\ge\epsilon
\end{align}
for all $j\in\N$. 
Let $\nu$ be a weak* limit point of the sequence
\begin{align*}
\br[\bigg]{\frac{1}{n(j)}\sum_{k=0}^{n(j)-1}\mu^{*k}*\delta_{x_j}}_j.
\end{align*}
Then $\nu$ is $\mu$-stationary, and a probability measure because of our recurrence assumption and the fact that all $x_j$ lie in the fixed compact set $K$ (Lemma~\ref{lem:non_escape}). 
But by unique ergodicity this forces $\nu=m_X$, contradicting \eqref{wrong_limit}. 
\end{proof}
\subsection{Lyapunov Functions \& Stronger Uniformity}\label{subsec:lyapunov}
Loosely speaking, \emph{\mbox{\textup{(}Foster--\textup{)}} Lyapunov functions} are functions enjoying certain contraction properties with respect to the random walk, to the effect that (on average) its dynamics are directed towards the \enquote{center} of the space, where the function takes values below some threshold. 
They were introduced into the study of random walks on homogeneous spaces by Eskin--Margulis~\cite{EM}, whose ideas were further developed by  Benoist--Quint~\cite{BQ_rec}. 
\begin{definition}\label{def:lyapunov}
A measurable function $V\colon X\to[0,\infty]$ is called a \emph{Lyapunov function} for the random walk on $X$ induced by $\mu$ if
\begin{enumerate}
\item[1.] it is proper, in the sense that the sublevel sets $V^{-1}([0,L])$ are relatively compact for $L\in[0,\infty)$, and 
\item[2.] there exist constants $\alpha<1$, $\beta\ge 0$ such that $\pi(\mu)V\le \alpha V+\beta$, where $\pi(\mu)$ is the convolution operator associated to $\mu$ introduced in~\S\ref{sec:spectral_gap}. 
\end{enumerate}
The inequality in the second condition above is referred to as the \emph{contraction property} of $V$.
\end{definition}
Allowing Lyapunov functions to take the value $\infty$ is conceptually important for the proofs of results such as Theorem~\ref{thm:BQ}, in order to show that the random walk does not accumulate near a lower-dimensional homogeneous subspace. Also, affording the possibility of non-continuous Lyapunov functions is crucial in recent constructions given in the literature~\cite{BQ_rec,expanding}. For the purposes of the discussion in this section, however, it is no big restriction to have in mind the case of a continuous Lyapunov function which is finite on all of $X$. 
\begin{remark}\label{rmk:lyapunov}
Let us collect some immediate observations about Lyapunov functions. 
\begin{enumerate}
\item If $V$ is a Lyapunov function, then so are $cV$ and $V+c$ for any constant $c>0$. 
In particular, one may impose an arbitrary lower bound on $V$, so that it is no restriction to assume that a Lyapunov function takes values $\ge 1$, say. 
\item Given a Lyapunov function $V'\colon X\to[0,\infty]$ for the $n_0$-step random walk (induced by the convolution power $\mu^{*n_0}$), one can construct a Lyapunov function $V$ for the random walk given by $\mu$ itself by setting 
\begin{align*}
V=\sum_{k=0}^{n_0-1}\alpha^{\frac{n_0-1-k}{n_0}}\pi(\mu)^kV'.
\end{align*}
\item By enlarging $\alpha$ and using properness, the contraction property in the definition of a Lyapunov function $V$ may be replaced by 
\begin{align*}
\pi(\mu)V\le \alpha V+\beta\mathds{1}_K
\end{align*}
for some compact $K\subset X$, where $\mathds{1}_K$ denotes the indicator function of $K$ (cf.\ \cite[Lemma~15.2.8]{MT}). \qedhere
\end{enumerate}
\end{remark}
Two examples in which a Lyapunov function exists are the following. 
\begin{example}[\cite{EM}]\label{ex:lyapunov_existence}
Identify $X=\SL_2(\R)/\SL_2(\Z)$ with the space of unimodular lattices in $\R^2$ as in Example~\ref{ex:lattice_height_fct} and recall that we denote by $\lambda_1(x)$ the length of a shortest non-zero vector in $x\in X$. 
Then for every compactly supported probability measure $\mu$ on $G$ whose support generates a Zariski dense subgroup there exist $\epsilon,\delta>0$ such that $V'=1+\epsilon\lambda_1^{-\delta}$ is a finite continuous Lyapunov function for the $n_0$-step random walk on $X$ induced by $\mu^{*n_0}$ for some $n_0\in\N$. 
This construction can be generalized to higher dimensions by taking into account the higher successive minima $\lambda_2,\dots,\lambda_d$ of lattices in $\R^d$. 
A more advanced construction also ensures existence of Lyapunov functions for Zariski dense probability measures with finite exponential moments when $G=\mathbb{G}(\R)$ is the group of real points of a Zariski connected semisimple algebraic group $\mathbb{G}$ defined over $\R$ such that $G$ has no compact factors. 
\end{example}
\begin{example}[{\cite{SW}}]\label{ex:SW}
Let $G=\SL_{d+1}(\R)$, $\Gamma=\SL_{d+1}(\Z)$ and $X=G/\Gamma$. 
For $0\le i\le m$ let $c_i>1$ be positive real numbers, $y_i\in \R^d$ vectors such that $y_0=0$ and $y_1,\dots,y_m$ span $\R^d$, $O_i\in\SO_d(\R)$ and set 
\begin{align*}
g_i=\begin{pmatrix}c_iO_i&y_i\\0&c_i^{-d}\end{pmatrix}\in G.
\end{align*}
Then for any choice of $p_0,\dots,p_m>0$ with $\sum_{i=0}^mp_i=1$, the measure $\mu=\sum_{i=0}^mp_i\delta_{g_i}$ defines a uniquely ergodic random walk on $X$ admitting a finite continuous Lyapunov function. 
\end{example}

It is well known that existence of a Lyapunov function implies recurrence properties of the random walk. 
\begin{lemma}[{\cite[Lemma~3.1]{EM}}]
Suppose the random walk on $X$ given by $\mu$ admits a finite continuous Lyapunov function $V$. 
Then this random walk is locally uniformly recurrent. 
\end{lemma}
The intuitive reason for this behavior is simple: 
The contraction property means that after a step of the random walk, the value of the Lyapunov function~$V$ on average gets smaller by a constant factor, at least when starting outside some compact set $K$ (cf.\ Remark~\ref{rmk:lyapunov}(iii) above), which one can think of as the \enquote{center} of the space. 
The set $K$ can be chosen as (closure of) a sublevel set of $V$. By the contraction property, the number of steps required to reach it is uniform over starting points $x$ in any given sublevel set of $V$, or in any given compact subset of $X$ in the case that $V$ is finite and continuous. 
This suggests that we might even let the starting points diverge, as long as this divergence is outcompeted by the geometric rate of contraction of $V$. 
We are led to the following notion of recurrence. 
\begin{definition}\label{def:Kn_uniform}
Let $(K_n)_n$ be a sequence of subsets of $X$. 
We say that the random walk on $X$ given by $\mu$ is \emph{$(K_n)_n$-uniformly recurrent} if for every $\epsilon>0$ there exists $n_0\in\N$ and a compact subset $M\subset X$ with 
\begin{align*}
\mu^{*n}*\delta_x(M)\ge 1-\epsilon
\end{align*}
for all $n\ge n_0$ and $x\in K_n$. 
It is called \emph{$(K_n)_n$-uniformly recurrent on average} if the above holds with the Ces\`aro averages $\frac{1}{n}\sum_{k=0}^{n-1}\mu^{*k}*\delta_x$ in place of $\mu^{*n}*\delta_x$. 
\end{definition}
\begin{remark}
We point out that contrary to the locally uniform situation, for the two versions of this property (with/without average) it is generally not clear whether one implies the other. 
\end{remark}
We are now going to establish such recurrence properties for certain families $(K_n)_n$ of sublevel sets of Lyapunov functions, which can be chosen to be increasing and to exhaust the part of $X$ where the Lyapunov function is finite.   
Recall that the \emph{Lyapunov exponent} of a function $\varphi\colon\N\to[1,\infty)$ is the exponential growth rate 
\begin{align*}
\lambda(\varphi)=\limsup_{n\to\infty}\tfrac1n\log\varphi(n).
\end{align*}
If $\lambda(\varphi)=0$, we say that $\varphi$ has \emph{sub-exponential growth}. 
\begin{proposition}\label{prop:Kn_uniform_rec}
Let $\varphi\colon\N\to[1,\infty)$ be a function. 
Suppose that the random walk on $X$ induced by $\mu$ admits a Lyapunov function $V$ with contraction factor $\alpha<1$ and set $K_n=V^{-1}([0,\varphi(n)])$. 
\begin{enumerate}
\item If $\varphi$ has Lyapunov exponent $\lambda(\varphi)<\log(\alpha^{-1})$, then the random walk on $X$ given by $\mu$ is $(K_n)_n$-uniformly recurrent. 
The number $n_0$ in the definition can be chosen independently of $\epsilon$. 
\item If $\varphi$ has sub-exponential growth, then the random walk on $X$ given by $\mu$ is $(K_n)_n$-uniformly recurrent on average. 
\end{enumerate}
\end{proposition}
The proof is a refinement of the methods in~\cite{BQ_rec,EM}. 
\begin{proof}
Let $\alpha,\beta$ be the constants from the contraction property of $V$ and define $B=\frac{\beta}{1-\alpha}$. 
We are going to use the same set $M$ for both parts of the proposition, namely $M=\overline{V^{-1}([0,2B/\epsilon])}$, which is compact since $V$ is proper. 
Then for $n\in\N$ and $x\in K_n$ we find, by repeatedly using the contraction property of $V$, 
\begin{align*}
\mu^{*n}*\delta_x(M^c)&\le \frac{\epsilon}{2B}\pi(\mu)^nV(x)\le \frac{\epsilon}{2B}(\alpha^nV(x)+B)\le \frac{\epsilon }{2B}\alpha^n\varphi(n)+\frac{\epsilon}{2}.
\end{align*}
When the exponential growth rate of $\varphi$ is less than $\log(\alpha^{-1})$, for some $n_0\in\N$ we have $\alpha^n\varphi(n)\le B$ for all $n\ge n_0$. 
This proves (i). 

In order to prove (ii) we use a similar estimate, but have to ensure that the values $\mu^{*k}*\delta_x(M^c)$ are small for a sufficiently large proportion of $0\le k<n$. 
For $x\in K_n$ we find, as above, 
\begin{align}\label{more_delicate_bound}
\mu^{*k}*\delta_x(M^c)\le \frac{\epsilon}{2B}\alpha^k\varphi(n)+\frac{\epsilon}{2}.
\end{align}
Using straightforward manipulations, we further see 
\begin{align*}
\alpha^k\varphi(n)\le B/2\iff\frac{k}{n}\ge \log(\alpha^{-1})^{-1}\br*{\frac1n\log\varphi(n)-\frac1n\log(B/2)},
\end{align*}
the right-hand side of which tends to $0$ as $n\to\infty$ by sub-exponential growth of $\varphi$. 
Hence, with $k(n)=\lfloor \epsilon n/4\rfloor$, we may choose $n_0$ large enough to ensure the above inequality holds for all $k\ge k(n)$ for $n\ge n_0$. 
For such $n$ we conclude, using \eqref{more_delicate_bound}, 
\begin{align*}
\frac1n\sum_{k=0}^{n-1}\mu^{*k}*\delta_x(M^c)&=\frac1n\sum_{k=0}^{k(n)-1}\mu^{*k}*\delta_x(M^c)+\frac1n\sum_{k=k(n)}^{n-1}\mu^{*k}*\delta_x(M^c)\\
&\le\frac{k(n)}{n}+\frac{3\epsilon}{4}\le \epsilon,
\end{align*}
which ends the proof of (ii). 
\end{proof}
Theorem~\ref{thm:uniform} can now be strengthened in the following way. 
\begin{theorem}\label{thm:lyapunov_uniform}
In addition to the assumptions of Theorem~\textup{\ref{thm:uniform}}, suppose that the random walk on $X$ induced by $\mu$ admits a Lyapunov function $V$. 
Let $\varphi\colon\N\to[1,\infty)$ have sub-exponential growth. 
Then for every $f\in C_c(X)$ we have 
\begin{align*}
\lim_{n\to\infty}\sup_{V(x)\le \varphi(n)}\abs[\bigg]{\frac1n\sum_{k=0}^{n-1}\int_Xf\dd(\mu^{*k}*\delta_x)-\int_Xf\dd m_X}=0.
\end{align*}
\end{theorem}
\begin{proof}
Using $(K_n)_n$-uniform recurrence on average for $K_n=V^{-1}([0,\varphi(n)])$ from Proposition~\ref{prop:Kn_uniform_rec}(ii), the proof of Theorem~\ref{thm:uniform} goes through with the obvious modifications. 
\end{proof}

\bibliographystyle{plain}
\bibliography{refs}

\end{document}